\documentclass[oneside,a4paper,11pt]{ctexart}
\usepackage{mathtools}
\usepackage{amsmath}
\allowdisplaybreaks[4]
\usepackage{amsthm}
\usepackage{amssymb}
\usepackage{mathrsfs}
\usepackage{bm}
\usepackage{fancyhdr}
\usepackage{geometry}
\usepackage{hyperref}
\usepackage{cite}

\newtheorem{thm}{Theorem}[section]
\newtheorem{cor}[thm]{Corollary}

\theoremstyle{definition}

\theoremstyle{remark}
\newtheorem{rem}[thm]{Remark}

\numberwithin{equation}{section}

\title{Minimal Norm Tensors Principle and Its Applications}
\author{Zhen Guo\footnote{Zhen Guo: Department of Mathematics, YunNan Normal University, 650500, Kunming, People's Republic of China. e-mail:gzh2001y@yahoo.com}.ShanLin Guan \footnote{ShanLin Guan: Department of Mathematics, YunNan Normal University, 650500, Kunming, People's Republic of China. e-mail:1923080002@user.ynnu.edu.cn.} \thanks{The authors are supported by the project (No.11531012) of NSFC.}}

\CTEXoptions[bibname={References}]
 
\date{}
\usepackage{fancyhdr}
\providecommand{\keywords}[1]
{
	\small	
	\textbf{Keywords:} #1
}
\providecommand{\classification}[1]
{
	\small	
	\textbf{Mathematics Subject Classification (2010):} #1
}

\begin{document}
	\maketitle
		\begin{abstract}
			In this paper we study the minimal norm tensors for general third covariant tensors and fourth covariant tensors, using this we can give a new explanation of Weyl tensor and Cotten tensor: Weyl tensor is the minimal norm tensor of Riemannian curvature tensor and Cotten tensor is the minimal norm tensor of divergence of Riemannian curvature tensor, and we also get some useful inequalities by computation the norm of minimal norm tensors.
		 
		\classification{53C24,53C25,53C21,53C80}	
			
	\end{abstract}
	
\section{Introduction} 
Let $V$ be a real $n-$dimensional vector space with a metric (inner product) $\langle \cdot, \cdot \rangle$, a standard basis of $V$ is denoted by $\left\{e_{i} \right\}$. Suppose $T$ is a third-order covariant tensor from space $V \times V \times V \to R$, we consider the following minimal values problem
\begin{align}
f(x^{1}_{1},x^{2}_{1},x^{3}_{1},x^{1}_{2},x^{2}_{2},x^{3}_{2},x^{1}_{3},x^{2}_{3},x^{3}_{3}) = \mid \mid F_{ijk} \mid \mid^{2} \label{eq1.1}
\end{align}
where
\begin{align}
F_{i_{1}i_{2}i_{3}} = {} & T_{i_{1}i_{2}i_{3}} + \frac{1}{2}(\sum_{\alpha=1}^{3} \sum_{\sigma \in \varphi(3)}x^{\sigma}_{\alpha} T^{\alpha}_{i_{\sigma(1)}} \delta_{i_{\sigma(2)}i_{\sigma(3)}}) \notag \\ = {} & T_{i_{1}i_{2}i_{3}}  + x^{1}_{1}T^{1}_{i_{1}}\delta_{i_{2}i_{3}} + x^{2}_{1}T^{1}_{i_{2}}\delta_{i_{1}i_{3}} + x^{3}_{1}T^{1}_{i_{3}}\delta_{i_{1}i_{2}} + x^{1}_{2}T^{2}_{i_{1}}\delta_{i_{2}i_{3}} \notag \\ + {} &  x^{2}_{2}T^{2}_{i_{2}}\delta_{i_{1}i_{3}} + x^{3}_{2}T^{2}_{i_{3}}\delta_{i_{1}i_{2}} + x^{1}_{3}T^{3}_{i_{1}}\delta_{i_{2}i_{3}} + x^{2}_{3}T^{3}_{i_{2}}\delta_{i_{1}i_{3}} + x^{3}_{3}T^{3}_{i_{3}}\delta_{i_{1}i_{2}}
\end{align}
where
\begin{align}
T_{i_{1}i_{2}i_{3}} = T(e_{i_{1}},e_{i_{2}},e_{i_{3}}), T^{1}_{i_{1}} = \sum_{i_{2} = i_{3} = 1}^{n} T_{i_{1}i_{2}i_{3}},T^{2}_{i_{2}} =  \sum_{i_{1} = i_{3} = 1}^{n} T_{i_{1}i_{2}i_{3}},T^{3}_{i_{3}} = \sum_{i_{1} = i_{2} = 1}^{n} T_{i_{1}i_{2}i_{3}}.
\end{align}
and $\varphi(3)$ is the third-order permutation group.
In this paper we study the question above at first, and we get the mimimal norm and minimal norm tensor as following.
\begin{thm}\label{thm1.1}
	Let $V$ be a real $n-$dimensional vector space with a metric (inner product) $\langle \cdot, \cdot \rangle$ and $T$ is a third-order covariant tensor, then the minimal norm tensor of equation \eqref{eq1.1} is
	\begin{align}
	F_{i_{1}i_{2}i_{3}} = {} & T_{i_{1}i_{2}i_{3}} - \frac{n+1}{(n-1)(n+2)}(T^{1}_{i_{1}}\delta_{i_{2}i_{3}} + T^{2}_{i_{2}}\delta_{i_{1}i_{3}} + T^{3}_{i_{3}}\delta_{i_{1}i_{2}}) \notag \\ + {} & \frac{1}{(n-1)(n+2)}(T^{1}_{i_{2}}\delta_{i_{1}i_{3}} + T^{1}_{i_{3}}\delta_{i_{1}i_{2}} + T^{2}_{i_{1}}\delta_{i_{2}i_{3}} + T^{2}_{i_{3}}\delta_{i_{1}i_{2}} + T^{3}_{i_{1}}\delta_{i_{2}i_{3}} + T^{3}_{i_{2}}\delta_{i_{1}i_{3}})
	\end{align}
	this is a trace-free tensor and the minimal norm is
	\begin{align}
	\mid \mid F_{i_{1}i_{2}i_{3}} \mid \mid^{2} = {} & \mid \mid T_{i_{1}i_{2}i_{3}} \mid \mid^{2} - \frac{n+1}{(n-1)(n+2)}(\mid \mid T^{1} \mid \mid^{2} + \mid \mid T^{2} \mid \mid^{2} + \mid \mid T^{3} \mid \mid^{2})  \notag \\ + {} & \frac{2}{(n-1)(n+2)}(\langle T^{1} , T^{2} \rangle + \langle T^{1} , T^{3} \rangle + \langle T^{2} , T^{3} \rangle)
	\end{align}
\end{thm}
\noindent Moreover, if we let $T_{i_{1}i_{2}i_{3}} = T_{ijk} = R_{lijk,l}$, then we have the following corollary
\begin{cor}\label{cor1.2}
	Let $M^{n}$ be a $n-$dimensional Riemannian manifold and $V=T_{p}M$ on any point $p \in M$, consider $T_{i_{1}i_{2}i_{3}} = T_{ijk} = R_{lijk,l}$ (in the sense of Levi-Civita connection), then the minimal norm tensor of equation \eqref{eq1.1} is Cotton tensor:
	\begin{align}
	F_{ijk} = C_{ijk} = R_{ik,j} - R_{ij,k} - \frac{1}{2(n-1)}(R_{,j}\delta_{ik} - R_{,k}\delta_{ij})
	\end{align}
\end{cor}
\noindent In the similar manner we consider $T$ is a fourth-order covariant tensor from space $V \times V \times V \times V \to R$. Under a standard basis $\left\{e_{i} \right\}$ of $V$, we have the following notations
\begin{align}
T_{i_{1}i_{2}i_{3}i_{4}} = T(e_{i_{1}},e_{i_{2}},e_{i_{3}},e_{i_{4}}), T^{34}_{i_{3}i_{4}} = T^{43}_{i_{3}i_{4}} = \sum^{n}_{i_{1}=i_{2}=1}T_{i_{1}i_{2}i_{3}i_{4}}, T^{24}_{i_{2}i_{4}} = T^{42}_{i_{2}i_{4}} =  \sum^{n}_{i_{1}=i_{3}=1}T_{i_{1}i_{2}i_{3}i_{4}}.
\end{align}
\begin{align}
T^{23}_{i_{2}i_{3}} = T^{32}_{i_{2}i_{3}} = \sum^{n}_{i_{1}=i_{4}=1}T_{i_{1}i_{2}i_{3}i_{4}},T^{14}_{i_{1}i_{4}} = T^{41}_{i_{1}i_{4}} = \sum^{n}_{i_{2} = i_{3} =1}T_{i_{1}i_{2}i_{3}i_{4}}, T^{13}_{i_{1}i_{3}} = \sum^{n}_{i_{2}=i_{4}=1}T_{i_{1}i_{2}i_{3}i_{4}}.
\end{align}
\begin{align}
T^{12}_{i_{1}i_{2}} = T^{21}_{i_{1}i_{2}} =  \sum^{n}_{i_{3}=i_{4}=1}T_{i_{1}i_{2}i_{3}i_{4}},
t^{12} = t^{21} = t^{34} = t^{43} = \sum^{n}_{i_{3}=i_{4}=1} T^{34}_{i_{3}i_{4}} = \sum^{n}_{i_{1}=i_{2}=1} T^{12}_{i_{1}i_{2}} .
\end{align}
\begin{align}
t^{13} = t^{31} = t^{24} = t^{42} =  \sum^{n}_{i_{2}=i_{4}=1} T^{24}_{i_{2}i_{4}} = \sum^{n}_{i_{1}=i_{3}=1} T^{13}_{i_{1}i_{3}}.
\end{align}
\begin{align}
t^{14} = t^{41} = t^{23} = t^{32} =   \sum^{n}_{i_{1}=i_{4}=1} T^{14}_{i_{1}i_{4}} = \sum^{n}_{i_{2}=i_{3}=1} T^{23}_{i_{2}i_{3}}.
\end{align}
Then we consider the following minimal values problem
\begin{align}
h = \mid \mid F_{i_{1}i_{2}i_{3}i_{4}} \mid \mid^{2} \label{eq1.11}
\end{align}
where
\begin{align}
F_{i_{1}i_{2}i_{3}i_{4}} = T_{i_{1}i_{2}i_{3}i_{4}} + {} & \frac{1}{4} \sum_{1 \le \alpha \neq \beta \le 4} \sum_{\sigma \in \varphi(4)} x^{\sigma(1)\sigma(2)}_{\alpha \beta} T^{\alpha \beta}_{i_{\sigma(1)}i_{\sigma(2)}} \delta_{i_{\sigma(3)}i_{\sigma(4)}} \notag \\ + {} & \frac{1}{4} \sum_{\sigma \in \varphi(4)} y^{\sigma(1)\sigma(2)} \delta_{i_{\sigma(1)}i_{\sigma(2)}} \delta_{i_{\sigma(3)}i_{\sigma(4)}}
\end{align}
and $\varphi(4)$ is the fourth-order permutation group. Then we have the following theorem.
\begin{thm}\label{thm1.3}
	Let $V$ be a real $n-$dimensional vector space with a metric (inner product) $\langle \cdot, \cdot \rangle$ and $T$ is a fourth-order covariant tensor, then the minimal norm tensor of equation \eqref{eq1.11} is
	\begin{align}
	F_{i_{1}i_{2}i_{3}i_{4}} = {} &  T_{i_{1}i_{2}i_{3}i_{4}} - \frac{2}{n(n-2)(n+4)} \sum_{\sigma \in \varphi^{\prime}_{1}(4)} T^{\sigma(3)\sigma(4)}_{i_{\sigma(1)}i_{\sigma(2)}} \delta_{i_{\sigma(3)}i_{\sigma(4)}} \notag \\ + {} & \frac{2}{n(n-2)(n+4)(n+2)}\sum_{\sigma \in \varphi^{\prime}_{1}(4)} T^{\sigma(3)\sigma(4)}_{i_{\sigma(4)}i_{\sigma(3)}} \delta_{i_{\sigma(1)}i_{\sigma(2)}} \notag \\ - {} & \frac{n^3+4n^2-4}{2n(n-2)(n+4)(n+2)}\sum_{\sigma \in \varphi^{\prime}_{1}(4)} T^{\sigma(3)\sigma(4)}_{i_{\sigma(3)}i_{\sigma(4)}}\delta_{i_{\sigma(1)}i_{\sigma(2)}} \notag \\ + {} & \frac{n+3}{(n-2)(n+4)(n+2)} \sum_{\sigma \in \varphi^{\prime}_{1}(4)}(T^{\sigma(3)\sigma(4)}_{i_{\sigma(3)}i_{\sigma(1)}}\delta_{i_{\sigma(2)}i_{\sigma(4)}} + T^{\sigma(3)\sigma(4)}_{i_{\sigma(1)}i_{\sigma(4)}}\delta_{i_{\sigma(2)}i_{\sigma(3)}}) \notag \\ - {} & \frac{1}{(n-2)(n+2)(n+4)}\sum_{\sigma \in \varphi^{\prime}_{1}(4)}(T^{\sigma(3)\sigma(4)}_{i_{\sigma(1)}i_{\sigma(3)}}\delta_{i_{\sigma(2)}i_{\sigma(4)}} + T^{\sigma(3)\sigma(4)}_{i_{\sigma(4)}i_{\sigma(1)}}\delta_{i_{\sigma(2)}i_{\sigma(3)}}) \notag \\ + {} &  \frac{n^2+3n+6}{4(n-2)(n+4)(n+2)(n-1)}\sum_{\sigma \in \varphi^{\prime}_{1}(4)}t^{\sigma(1)\sigma(2)}\delta_{i_{\sigma(1)}i_{\sigma(2)}} \delta_{i_{\sigma(3)}i_{\sigma(4)}} \notag \\ - {} &  \frac{3n+2}{2(n-2)(n+4)(n+2)(n-1)}\sum_{\sigma \in \varphi^{\prime}_{1}(4)}t^{\sigma(1)\sigma(3)}\delta_{i_{\sigma(1)}i_{\sigma(2)}} \delta_{i_{\sigma(3)}i_{\sigma(4)}} \label{eq1.14}
	\end{align}
	where $\varphi^{\prime}_{1}(4) = \lbrace \sigma \in \varphi(4) \mid \sigma(3) < \sigma(4) \rbrace$. 	This is a trace-free tensor and the minimal norm is \begin{align}
	\mid \mid F_{i_{1}i_{2}i_{3}i_{4}} \mid \mid^{2} = {} &  \mid \mid T_{i_{1}i_{2}i_{3}i_{4}} \mid \mid^{2} - \frac{2}{n(n-2)(n+4)} \sum_{\sigma \in \varphi^{\prime}_{2}(4)} T^{\sigma(3)\sigma(4)}_{i_{\sigma(1)}i_{\sigma(2)}} (T^{\sigma(1)\sigma(2)}_{i_{\sigma(1)}i_{\sigma(2)}} + T^{\sigma(1)\sigma(2)}_{i_{\sigma(2)}i_{\sigma(1)}}) \notag \\ + {} & \frac{4}{n(n-2)(n+4)(n+2)}\sum_{\sigma \in \varphi^{\prime}_{2}(4)} T^{\sigma(3)\sigma(4)}_{i_{\sigma(4)}i_{\sigma(3)}} T^{\sigma(3)\sigma(4)}_{i_{\sigma(3)}i_{\sigma(4)}} \notag \\ - {} & \frac{n^3+4n^2-4}{2n(n-2)(n+4)(n+2)}\sum_{\sigma \in \varphi^{\prime}_{1}(4)} T^{\sigma(3)\sigma(4)}_{i_{\sigma(3)}i_{\sigma(4)}}T^{\sigma(3)\sigma(4)}_{i_{\sigma(3)}i_{\sigma(4)}} \notag \\ + {} & \frac{2(n+3)}{(n-2)(n+4)(n+2)} (\sum_{\sigma \in \varphi^{\prime}_{3}(4)}T^{\sigma(3)\sigma(4)}_{i_{\sigma(3)}i_{\sigma(1)}}T^{\sigma(1)\sigma(3)}_{i_{\sigma(1)}i_{\sigma(3)}} + \sum_{\sigma \in \varphi^{\prime}_{4}(4)} T^{\sigma(3)\sigma(4)}_{i_{\sigma(1)}i_{\sigma(4)}} T^{\sigma(1)\sigma(4)}_{i_{\sigma(1)}i_{\sigma(4)}}) \notag \\ - {} & \frac{2}{(n-2)(n+2)(n+4)}(\sum_{\sigma \in \varphi^{\prime}_{3}(4)}T^{\sigma(3)\sigma(4)}_{i_{\sigma(1)}i_{\sigma(3)}}T^{\sigma(1)\sigma(3)}_{i_{\sigma(1)}i_{\sigma(3)}} + \sum_{\sigma \in \varphi^{\prime}_{4}(4)} T^{\sigma(3)\sigma(4)}_{i_{\sigma(4)}i_{\sigma(1)}}T^{\sigma(1)\sigma(4)}_{i_{\sigma(1)}i_{\sigma(4)}}) \notag \\ + {} &  \frac{n^2+3n+6}{4(n-2)(n+4)(n+2)(n-1)}\sum_{\sigma \in \varphi^{\prime}_{1}(4)}(t^{\sigma(1)\sigma(2)})^{2} \notag \\ - {} &  \frac{3n+2}{2(n-2)(n+4)(n+2)(n-1)}\sum_{\sigma \in \varphi^{\prime}_{1}(4)}t^{\sigma(1)\sigma(3)}t^{\sigma(1)\sigma(2)}
	\end{align}
	where $\varphi^{\prime}_{2}(4) = \lbrace \sigma \in \varphi(4) \mid \sigma(3) < \sigma(4), \sigma(1) < \sigma(2) \rbrace$, $\varphi^{\prime}_{3}(4) = \lbrace \sigma \in \varphi(4) \mid \sigma(3) < \sigma(4), \sigma(1) < \sigma(3) \rbrace$ and $\varphi^{\prime}_{4}(4) = \lbrace \sigma \in \varphi(4) \mid \sigma(3) < \sigma(4), \sigma(1) < \sigma(4) \rbrace$.
\end{thm}
\noindent In particular, if $T^{\alpha \beta}_{i_{1}i_{2}}$ are symmetric i.e. $T^{\alpha \beta}_{i_{1}i_{2}} = T^{\alpha \beta}_{i_{2}i_{1}}$ for any $1 \le \alpha \neq \beta \le 4$ , then we have the following corollary.
\begin{cor}
	Let $V$ be a real $n-$dimensional vector space with a metric (inner product) $\langle \cdot, \cdot \rangle$ and $T$ is a fourth-order covariant tensor with symmetric traces in any two indices, then the minimal norm tensor of equation \eqref{eq1.11} is
	\begin{align}
	F_{i_{1}i_{2}i_{3}i_{4}} = {} &  T_{i_{1}i_{2}i_{3}i_{4}} - \frac{1}{n(n-2)(n+4)} \sum_{\sigma \in \varphi(4)} T^{\sigma(3)\sigma(4)}_{i_{\sigma(1)}i_{\sigma(2)}} \delta_{i_{\sigma(3)}i_{\sigma(4)}} \notag \\ - {} & \frac{n^2+2n-4}{4n(n-2)(n+4)}\sum_{\sigma \in \varphi(4)} T^{\sigma(3)\sigma(4)}_{i_{\sigma(3)}i_{\sigma(4)}} \delta_{i_{\sigma(1)}i_{\sigma(2)}}  \notag \\ + {} & \frac{1}{(n-2)(n+4)} \sum_{\sigma \in \varphi(4)}T^{\sigma(3)\sigma(4)}_{i_{\sigma(3)}i_{\sigma(1)}}\delta_{i_{\sigma(2)}i_{\sigma(4)}}  \notag \\ + {} &  \frac{n^2+3n+6}{8(n-2)(n+4)(n+2)(n-1)}\sum_{\sigma \in \varphi(4)}t^{\sigma(1)\sigma(2)}\delta_{i_{\sigma(1)}i_{\sigma(2)}} \delta_{i_{\sigma(3)}i_{\sigma(4)}} \notag \\ - {} &  \frac{3n+2}{4(n-2)(n+4)(n+2)(n-1)}\sum_{\sigma \in \varphi(4)}t^{\sigma(1)\sigma(3)}\delta_{i_{\sigma(1)}i_{\sigma(2)}} \delta_{i_{\sigma(3)}i_{\sigma(4)}}
	\end{align}
	This is a trace-free tensor and the minimal norm is
	\begin{align}
	\mid \mid F_{i_{1}i_{2}i_{3}i_{4}} \mid \mid^{2}= {} &  \mid \mid T_{i_{1}i_{2}i_{3}i_{4}} \mid \mid^{2} - \frac{1}{n(n-2)(n+4)} \sum_{\sigma \in \varphi(4)} T^{\sigma(3)\sigma(4)}_{i_{\sigma(1)}i_{\sigma(2)}}T^{\sigma(1)\sigma(2)}_{i_{\sigma(1)}i_{\sigma(2)}}  \notag \\ - {} & \frac{n^2+2n-4}{4n(n-2)(n+4)}\sum_{\sigma \in \varphi(4)} T^{\sigma(3)\sigma(4)}_{i_{\sigma(3)}i_{\sigma(4)}} T^{\sigma(3)\sigma(4)}_{i_{\sigma(3)}i_{\sigma(4)}}   \notag \\ + {} & \frac{1}{(n-2)(n+4)} \sum_{\sigma \in \varphi(4)}T^{\sigma(3)\sigma(4)}_{i_{\sigma(3)}i_{\sigma(1)}}T^{\sigma(1)\sigma(3)}_{i_{\sigma(3)}i_{\sigma(1)}}  \notag \\ + {} &  \frac{n^2+3n+6}{8(n-2)(n+4)(n+2)(n-1)}\sum_{\sigma \in \varphi(4)}(t^{\sigma(1)\sigma(2)})^{2} \notag \\ - {} &  \frac{3n+2}{4(n-2)(n+4)(n+2)(n-1)}\sum_{\sigma \in \varphi(4)}t^{\sigma(1)\sigma(3)}t^{\sigma(1)\sigma(2)}
	\end{align}
	In particular, if   $T_{i_{1}i_{2}i_{3}i_{4}}$ is totally symmetric for any two indices,  then the minimal norm tensor of equation \eqref{eq1.11} is
	\begin{align}
	F_{i_{1}i_{2}i_{3}i_{4}} = {} & T_{i_{1}i_{2}i_{3}i_{4}} - \frac{1}{4(n+4)}\sum_{\sigma \in \varphi(4)} T_{i_{\sigma(1)}i_{\sigma(2)}} \delta_{i_{\sigma(3)}i_{\sigma(4)}} \notag \\ + {} & \frac{T}{8(n+4)(n+2)}\sum_{\sigma \in \varphi(4)} \delta_{i_{\sigma(1)}i_{\sigma(2)}} \delta_{i_{\sigma(3)}i_{\sigma(4)}}
	\end{align}
	where $T_{i_{1}i_{2}} = \sum^{n}_{i_{3} = i_{4} = 1} T_{i_{1}i_{2}i_{3}i_{4}}$ and $T = \sum^{n}_{i_{1} = i_{2} =1}T_{i_{1}i_{2}}$. This is a trace-free tensor and the minimal norm is
	\begin{align}
	\mid \mid F_{i_{1}i_{2}i_{3}i_{4}} \mid \mid^{2} = \mid \mid T_{i_{1}i_{2}i_{3}i_{4}} \mid \mid^{2} -  \frac{6}{n+4}\mid \mid T_{i_{1}i_{2}} \mid \mid^{2} + \frac{3T^{2}}{(n+4)(n+2)}
	\end{align}
\end{cor}

\noindent Moreover, if we let $T _{i_{1}i_{2}i_{3}i_{4}} = T_{ijkl} = R_{ijkl}$, then we have the following corollary
\begin{cor}\label{cor1.5}
	Let $M^{n}$ be a $n-$dimensional Riemannian manifold and $V=T_{p}M$ on any point $p \in M$, consider $T _{i_{1}i_{2}i_{3}i_{4}} = T_{ijkl} = R_{ijkl}$, then the minimal norm tensor of equation \eqref{eq1.11} is Weyl conformal curvature tensor:
	\begin{align}
	F_{ijkl} = W_{ijkl} = {} & R_{ijkl}  - \frac{1}{n-2}[R_{jl}\delta_{ik} - R_{jk}\delta_{il} - R_{il}\delta_{jk} + R_{ik}\delta_{jl}] \notag \\ & + \frac{R}{(n-1)(n-2)}(\delta_{jl}\delta_{ik} - \delta_{jk}\delta_{il})
	\end{align}
\end{cor}

\begin{rem}
	There are many researchers study the trace-free decomposition tensors. J.W.Alexander research the decomposition of tensor  in  \cite{alexander1926decomposition}, D.Krupka research the trace-free decomposition in  \cite{krupka1995trace,krupka1996trace,krupka2003weyl,krupka2006trace,krupka2017decompositions} and he get some formulaes of trace-free decomposition for some special type of tensors, in particular, he get a nice corollary that the  trace-free decomposition tensor of (2,2) and (1,3) Riemannian curvature tensor are (2,2) Weyl conformal curvature tensor and (1,3) projective curvature tensor, respectively. And the trace in \cite{krupka1995trace,krupka1996trace,krupka2003weyl,krupka2006trace,krupka2017decompositions} means contraction upper indices with lower indices(In our paper trace means contraction with all indices). J.Mikes, M.Jukl, L.Juklov{\'a}, L.Lakom{\'a} generalize some results of D.Krupka and give some applications in \cite{mikes1996general,jukl2013decomposition,lakoma2004decomposition,mikes2011some}. M.Cr{\^a}{\c{s}}m{\u{a}}reanu research trace decomposition in recurrence problems and applications to almost projective invariants in \cite{cracsmuareanu2001particular,crasmareanu2001trace}. V.T.Toth and S.G.Turyshev research the trace-free decomposition of $(0,3)$ tensor and totally symmetric tensor of $(0,r)$ from the point of physics in \cite{toth2021efficient}.
\end{rem}

\section{Proof of Main Theorems}
\begin{proof}[Proof of theorem \ref{thm1.1}] 		By a direct computation we have
	\begin{align}
	\frac{\partial F_{i_{1}i_{2}i_{3}}}{\partial x^{\gamma}_{\beta}} = {} & \frac{1}{2} \sum_{\alpha=1}^{3} \sum_{\sigma \in \varphi(3)} \delta_{\gamma}^{\sigma(1)} \delta_{\beta}^{\alpha} T^{\alpha}_{i_{\sigma(1)}} \delta_{i_{\sigma(2)}i_{\sigma(3)}} \notag \\ = {} & \delta^{1}_{\gamma}T^{\beta}_{i_{1}}\delta_{i_{2}i_{3}} + \delta^{2}_{\gamma} T^{\beta}_{i_{2}} \delta_{i_{1}i_{3}} + \delta^{3}_{\gamma} T^{\beta}_{i_{3}} \delta_{i_{1}i_{2}}.
	\end{align}
	Next we compute the trace of $F_{i_{1}i_{2}i_{3}}$
	\begin{align}
	F^{1}_{i_{1}} = {} & \sum^{n}_{i_{2}=i_{3} = 1}  (T_{i_{1}i_{2}i_{3}} + \frac{1}{2}(\sum_{\alpha=1}^{3} \sum_{\sigma \in \varphi(3)}x^{\sigma}_{\alpha} T^{\alpha}_{i_{\sigma(1)}} \delta_{i_{\sigma(2)}i_{\sigma(3)}}) ) \notag \\ = {} & T^{1}_{i_{1}} + \sum_{\alpha=1}(n x^{1}_{\alpha} + x^{2}_{\alpha} + x^{3}_{\alpha})T^{\alpha}_{i_{1}} \notag \\ = {} &  T^{1}_{i_{1}} + \sum_{\alpha=1} y^{1}_{\alpha}T^{\alpha}_{i_{1}}.
	\end{align}
	In the similar manner we have
	\begin{align}
	F^{2}_{i_{2}} = T^{2}_{i_{2}} +  \sum_{\alpha=1} (x^{1}_{\alpha} + n x^{2}_{\alpha} + x^{3}_{\alpha})T^{\alpha}_{i_{2}} =  T^{2}_{i_{2}} +  \sum_{\alpha=1} y^{2}_{\alpha}T^{\alpha}_{i_{2}}.
	\end{align}
	\begin{align}
	F^{3}_{i_{3}} = T^{3}_{i_{3}} + \sum_{\alpha=1} (x^{1}_{\alpha} + x^{2}_{\alpha} + n x^{3}_{\alpha})T^{\alpha}_{i_{3}} =  T^{3}_{i_{3}} +  \sum_{\alpha=1} y^{3}_{\alpha}T^{\alpha}_{i_{3}}.
	\end{align}
	Then we have
	\begin{align}
	\frac{1}{2}\frac{\partial f}{\partial x^{\gamma}_{\beta}} = {} & \langle \frac{\partial  F_{i_{1}i_{2}i_{3}}}{\partial x^{\gamma}_{\beta}}, F_{i_{1}i_{2}i_{3}} \rangle = \delta^{1}_{\gamma} T^{\beta}_{i_{1}} F^{1}_{i_{1}} +  \delta^{2}_{\gamma} T^{\beta}_{i_{2}} F^{2}_{i_{2}} + \delta^{3}_{\gamma} T^{3}_{i_{3}} T^{3}_{i_{3}} \notag \\ = {} & \sum_{\alpha=1} \delta^{\alpha}_{\gamma} \langle T^{\beta},F^{\alpha} \rangle = \langle T^{\beta}, F^{\gamma}\rangle =  \langle T^{\beta}, T^{\gamma} + \sum_{\alpha=1}y^{\gamma}_{\alpha} T^{\alpha} \rangle \notag \\ = {} &  \sum_{\alpha=1} \langle T^{\beta}, \delta^{\gamma}_{\alpha}T^{\alpha} + y^{\gamma}_{\alpha} T^{\alpha} \rangle = (\delta^{\gamma}_{\alpha} + y^{\gamma}_{\alpha}) \langle T^{\beta},T^{\alpha} \rangle = 0.
	\end{align}
	Or equivalently
	\begin{align}
	(I + Y)A = (I + BX) A = O.
	\end{align}
	where $a_{\alpha \beta} = \langle T^{\alpha},T^{\beta} \rangle$ , $O$ is the zero matrix of $3 \times 3$ whose elements all are zero and
	\begin{align}
	B =
	\begin{bmatrix}
	n & 1 & 1 \\
	1 & n & 1 \\
	1 & 1 & n
	\end{bmatrix}
	\end{align}
	such that $Y = BX$. \\
	1) If $\det(A) \neq 0$, then $X = - B^{-1}$ and
	the minimal norm tensor (It is easy to see the tensor we solved is minimal) is
	\begin{align}
	F_{i_{1}i_{2}i_{3}} = {} & T_{i_{1}i_{2}i_{3}} - \frac{n+1}{(n-1)(n+2)}(T^{1}_{i_{1}}\delta_{i_{2}i_{3}} + T^{2}_{i_{2}}\delta_{i_{1}i_{3}} + T^{3}_{i_{3}}\delta_{i_{1}i_{2}}) \notag \\ + {} & \frac{1}{(n-1)(n+2)}(T^{1}_{i_{2}}\delta_{i_{1}i_{3}} + T^{1}_{i_{3}}\delta_{i_{1}i_{2}} + T^{2}_{i_{1}}\delta_{i_{2}i_{3}} + T^{2}_{i_{3}}\delta_{i_{1}i_{2}} + T^{3}_{i_{1}}\delta_{i_{2}i_{3}} + T^{3}_{i_{2}}\delta_{i_{1}i_{3}}) \label{eq2.8}
	\end{align}
	This is a trace-free tensor and by a direct computation the minimal norm is
	\begin{align}
	\mid \mid F_{i_{1}i_{2}i_{3}} \mid \mid^{2} = {} & \mid \mid T_{i_{1}i_{2}i_{3}} \mid \mid^{2} - \frac{n+1}{(n-1)(n+2)}(\mid \mid T^{1} \mid \mid^{2} + \mid \mid T^{2} \mid \mid^{2} + \mid \mid T^{3} \mid \mid^{2})  \notag \\ + {} & \frac{2}{(n-1)(n+2)}(\langle T^{1} , T^{2} \rangle + \langle T^{1} , T^{3} \rangle + \langle T^{2} , T^{3} \rangle)
	\end{align}
	2) If $\det(A) = 0$, noting that $A$ is the metric matrix of $T^{1},T^{2},T^{3}$, then $T^{1},T^{2},T^{3}$ are dependent linearly.  In this case the effect variables are less than before since we can use the relation of linear dependence, we can prove that the minimal norm tensor solved by the similar disscussion in $\det(A) \neq 0$  can be written in the form of equation \eqref{eq2.8} if we use  the relation of linear dependent in  equation \eqref{eq2.8}.  This proves the theorem.
\end{proof}

\begin{proof}[Proof of corollary \ref{cor1.2}]
	Letting $T_{i_{1}i_{2}i_{3}} = T_{ijk} = R_{lijk,l}$, then by a direct computation we have
	\begin{align}
	T_{ijk} = R_{lijk,l} = R_{ik,j} - R_{ij,k}
	\end{align}
	and \begin{align}
	T^{1}_{i} = 0, T^{2}_{j} = \frac{R_{,j}}{2}, T^{3}_{k} = - \frac{R_{,k}}{2}.
	\end{align}
	where we used the second Bianchi identity. Then \begin{align}
	F_{ijk} = R_{ik,j} - R_{ij,k} - \frac{1}{2(n-1)}(R_{,j}\delta_{ik} - R_{,k}\delta_{ij})
	\end{align}
	This is Cotton tensor $C_{ijk}$. This proves the corollary.
\end{proof}

\begin{proof}[Proof of theorem \ref{thm1.3}]
	By a direct computation we have
	\begin{align}
	\frac{\partial F_{i_{1}i_{2}i_{3}i_{4}}}{\partial x_{\alpha \beta}^{\tau(1)\tau(2)}} = \frac{1}{4} \sum_{\sigma \in \varphi(4)} \delta_{\sigma(1)}^{\tau(1)}\delta_{\sigma(2)}^{\tau(2)} T^{\alpha \beta}_{i_{\sigma(1)}i_{\sigma(2)}} \delta_{i_{\sigma(3)}i_{\sigma(4)}} = \frac{1}{2} T^{\alpha \beta}_{i_{\tau(1)}i_{\tau(2)}} \delta_{i_{\tau(3)}i_{\tau(4)}}
	\end{align}
	\begin{align}
	\frac{\partial F_{i_{1}i_{2}i_{3}i_{4}}}{\partial y^{\tau(1)\tau(2)}} = \frac{1}{4} \sum_{\sigma \in \varphi(4)} \delta_{\sigma(1)}^{\tau(1)}\delta_{\sigma(2)}^{\tau(2)} \delta_{i_{\sigma(1)}i_{\sigma(2)}} \delta_{i_{\sigma(3)}i_{\sigma(4)}} = \frac{1}{2}\delta_{i_{\tau(1)}i_{\tau(2)}} \delta_{i_{\tau(3)}i_{\tau(4)}}
	\end{align}
	The critical point equations are
	\begin{align}
	\frac{\partial \mid \mid F_{i_{1}i_{2}i_{3}i_{4}} \mid \mid^{2}}{\partial x_{\alpha \beta}^{\tau(1)\tau(2)}} = 2 F_{i_{1}i_{2}i_{3}i_{4}} \frac{\partial F_{i_{1}i_{2}i_{3}i_{4}}}{\partial x_{\alpha \beta}^{\tau(1)\tau(2)}} = 0, \frac{\partial \mid \mid F_{i_{1}i_{2}i_{3}i_{4}} \mid \mid^{2}}{\partial y^{\tau(1)\tau(2)}} = 2 F_{i_{1}i_{2}i_{3}i_{4}} \frac{\partial F_{i_{1}i_{2}i_{3}i_{4}}}{\partial y^{\tau(1)\tau(2)}} = 0
	\end{align}
	It is easy to see that
	\begin{align}
	\frac{\partial \mid \mid F_{i_{1}i_{2}i_{3}i_{4}} \mid \mid^{2}}{\partial y^{\tau(1)\tau(2)}} = 0 \Longleftrightarrow \sum_{i_{1} = i_{2}=1}F^{\tau(1)\tau(2)}_{i_{1}i_{2}} = \langle I, F^{\tau(1)\tau(2)} \rangle = 0
	\end{align}
	If $\tau(1) < \tau(2)$, then
	\begin{align}
	\frac{\partial \mid \mid F_{i_{1}i_{2}i_{3}i_{4}} \mid \mid^{2}}{\partial x_{\alpha \beta}^{\tau(1)\tau(2)}} = 0 \Longleftrightarrow T^{\alpha \beta}_{i_{\tau(1)}i_{\tau(2)}} F^{\tau(1)\tau(2)}_{i_{\tau(1)}i_{\tau(2)}} = \langle T^{\alpha \beta}, F^{\tau(1)\tau(2)} \rangle = 0.
	\end{align}
	If $\tau(2) < \tau(1)$, then
	\begin{align}
	\frac{\partial \mid \mid F_{i_{1}i_{2}i_{3}i_{4}} \mid \mid^{2}}{\partial x_{\alpha \beta}^{\tau(1)\tau(2)}} = 0 \Longleftrightarrow T^{\alpha \beta}_{i_{\tau(2)}i_{\tau(1)}} F^{\tau(1)\tau(2)}_{i_{\tau(1)}i_{\tau(2)}} = \langle (T^{\alpha \beta})^{T}, F^{\tau(1)\tau(2)} \rangle = 0.
	\end{align}
	where $(T^{\alpha \beta})^{T}_{i_{\tau(1)}i_{\tau(2)}} =T_{i_{\tau(2)}i_{\tau(1)}}^{\alpha \beta}$. Next we compute $F^{\tau(1)\tau(2)}$ under assumption that $\tau(1) < \tau(2)$ (since $F^{\tau(1)\tau(2)} = F^{\tau(2)\tau(1)}$), by a direct computation we have \pagebreak[3]
	\begin{align}
	F^{\tau(1)\tau(2)}_{i_{\tau(1)} i_{\tau(2)}} = {} &  \delta_{i_{\tau(3)}i_{\tau(4)}} F_{i_{1}i_{2}i_{3}i_{4}} \notag \\ = {} & T^{\tau(1)\tau(2)}_{i_{\tau(1)} i_{\tau(2)}} + \frac{1}{2}(n x^{\tau(1)\tau(2)}_{\alpha \beta} + x^{\tau(1)\tau(3)}_{\alpha \beta} + x^{\tau(1)\tau(4)}_{\alpha \beta} + x^{\tau(3)\tau(2)}_{\alpha \beta} + x^{\tau(4)\tau(2)}_{\alpha \beta})T^{\alpha \beta}_{i_{\tau(1)} i_{\tau(2)}} \notag \\ + {} & \frac{1}{2}(n x^{\tau(2)\tau(1)}_{\alpha \beta} + x^{\tau(3)\tau(1)}_{\alpha \beta} + x^{\tau(4)\tau(1)}_{\alpha \beta} + x^{\tau(2)\tau(3)}_{\alpha \beta} + x^{\tau(2)\tau(4)}_{\alpha \beta})T^{\alpha \beta}_{i_{\tau(2)} i_{\tau(1)}} \notag \\ + {} & 2(ny^{\tau(1)\tau(2)} + y^{\tau(1)\tau(3)} + y^{\tau(1)\tau(4)} + \frac{1}{4}(x^{\tau(3)\tau(4)}_{\alpha \beta} + x^{\tau(4)\tau(3)}_{\alpha \beta})t^{\alpha \beta})\delta_{i_{\tau(1)}i_{\tau(2)}}\notag \\ = {} & T^{\tau(1)\tau(2)}_{i_{\tau(1)} i_{\tau(2)}} + \sum_{\alpha < \beta}(n x^{\tau(1)\tau(2)}_{\alpha \beta} + x^{\tau(1)\tau(3)}_{\alpha \beta} + x^{\tau(1)\tau(4)}_{\alpha \beta} + x^{\tau(3)\tau(2)}_{\alpha \beta} + x^{\tau(4)\tau(2)}_{\alpha \beta})T^{\alpha \beta}_{i_{\tau(1)} i_{\tau(2)}} \notag \\ + {} & \sum_{\alpha < \beta} (n x^{\tau(2)\tau(1)}_{\alpha \beta} + x^{\tau(3)\tau(1)}_{\alpha \beta} + x^{\tau(4)\tau(1)}_{\alpha \beta} + x^{\tau(2)\tau(3)}_{\alpha \beta} + x^{\tau(2)\tau(4)}_{\alpha \beta})T^{\alpha \beta}_{i_{\tau(2)} i_{\tau(1)}} \notag \\ + {} & 2(ny^{\tau(1)\tau(2)} + y^{\tau(1)\tau(3)} + y^{\tau(1)\tau(4)} + \frac{1}{4}(x^{\tau(3)\tau(4)}_{\alpha \beta} + x^{\tau(4)\tau(3)}_{\alpha \beta})t^{\alpha \beta})\delta_{i_{\tau(1)}i_{\tau(2)}}
	\end{align}
	For convenience we denote that
	\begin{align}
	\bar{x}^{\tau(1)\tau(2)}_{\alpha \beta} = n x^{\tau(1)\tau(2)}_{\alpha \beta} + x^{\tau(1)\tau(3)}_{\alpha \beta} + x^{\tau(1)\tau(4)}_{\alpha \beta} + x^{\tau(3)\tau(2)}_{\alpha \beta} + x^{\tau(4)\tau(2)}_{\alpha \beta}
	\end{align}
	\begin{align}
	\bar{x}^{\tau(2)\tau(1)}_{\alpha \beta} = n x^{\tau(2)\tau(1)}_{\alpha \beta} + x^{\tau(3)\tau(1)}_{\alpha \beta} + x^{\tau(4)\tau(1)}_{\alpha \beta} + x^{\tau(2)\tau(3)}_{\alpha \beta} + x^{\tau(2)\tau(4)}_{\alpha \beta}
	\end{align}
	\begin{align}
	\bar{y}^{\tau(1)\tau(2)} = 2(ny^{\tau(1)\tau(2)} + y^{\tau(1)\tau(3)} + y^{\tau(1)\tau(4)} + \frac{1}{4}(x^{\tau(3)\tau(4)}_{\alpha \beta} + x^{\tau(4)\tau(3)}_{\alpha \beta})t^{\alpha \beta})
	\end{align}
	where $\tau(1) < \tau(2)$. Then we can rewrite $F^{\tau(1)\tau(2)}$ as following
	\begin{align}
	F^{\tau(1)\tau(2)} = T^{\tau(1)\tau(2)} + \sum_{\alpha < \beta} \bar{x}^{\tau(1)\tau(2)}_{\alpha \beta} T^{\alpha \beta} + \bar{x}^{\tau(2)\tau(1)}_{\alpha \beta} (T^{\alpha \beta})^{T}
	\end{align}
	Combining with the critical point equations
	\begin{align}
	\langle (T^{\alpha \beta})^{T}, F^{\tau(1)\tau(2)} \rangle = 0, \langle T^{\alpha \beta}, F^{\tau(1)\tau(2)} \rangle = 0, \langle I, F^{\tau(1)\tau(2)} \rangle = 0.
	\end{align}
	1) If $\lbrace T^{\alpha \beta},(T^{\alpha \beta})^{T}, I \rbrace$  (where $(1 \le \alpha < \beta) \le 4$) are linearly independent, then the determinate of inner product matrix is positive. By cramer rule we know $F^{\tau(1) \tau(2)} = 0$ since all coefficients are zero. Then we have
	\begin{align}
	1 + \bar{x}^{\tau(1)\tau(2)}_{\tau(1)\tau(2)} + \bar{x}^{\tau(2)\tau(1)}_{\tau(1)\tau(2)} = 0, \tau(1) < \tau(2).
	\end{align}
	\begin{align}
	\bar{x}^{\tau(1)\tau(2)}_{\alpha \beta} + \bar{x}^{\tau(2)\tau(1)}_{\alpha \beta} = 0, \alpha < \beta, \alpha \beta \neq \tau(1) \tau(2),\tau(1) <  \tau(2).
	\end{align}
	Then
	we can rewrite it by the definition of $\bar{x}$ that
	\begin{align}
	Z =\begin{bmatrix}
	n & 1 & 1 & 1 & 1 & 0 \\ 1 & n & 1 & 1 & 0 & 1 \\ 1 & 1 & n & 0 & 1 & 1 \\ \vdots & \vdots & \cdots & \cdots & \cdots & \vdots \\ 0 & 1 & 1 & & 1 & n
	\end{bmatrix} 	\begin{bmatrix}
	x^{12}_{12} + x^{21}_{12} & x^{12}_{13} + x^{21}_{13} & \cdots &  x^{12}_{34} + x^{21}_{34} \\
	x^{13}_{12} + x^{31}_{12} & x^{1
		3}_{13} + x^{31}_{13} & \cdots &  x^{13}_{34} + x^{31}_{43} \\ \vdots & \vdots & \cdots & \vdots \\
	x^{34}_{12} + x^{43}_{12} & x^{34}_{13} + x^{43}_{13} & \cdots &  x^{34}_{34} + x^{43}_{34}
	\end{bmatrix} = C \tilde{X} =-I
	\end{align}
	where
	\begin{align}
	Z =
	\begin{bmatrix}
	\bar{x}^{12}_{12} + \bar{x}^{21}_{12} & \bar{x}^{12}_{13} + \bar{x}^{21}_{13} & \cdots &  \bar{x}^{12}_{34} + \bar{x}^{21}_{34} \\
	\bar{x}^{13}_{12} + \bar{x}^{31}_{12} & \bar{x}^{1
		3}_{13} + \bar{x}^{31}_{13} & \cdots &  \bar{x}^{13}_{34} + \bar{x}^{31}_{43} \\ \vdots & \vdots & \cdots & \vdots \\
	\bar{x}^{34}_{12} + \bar{x}^{43}_{12} & \bar{x}^{34}_{13} + \bar{x}^{43}_{13} & \cdots &  \bar{x}^{34}_{34} + \bar{x}^{43}_{34}
	\end{bmatrix}
	\end{align}
	We can solve it as following
	\begin{align}
	\tilde{X} =	\frac{-1}{n^{2}+2 n-8}	\begin{bmatrix}
	\frac{n^{2}+2 n-4}{n} & -1 & -1 & -1 & -1 & \frac{4}{n} \\
	-1 & \frac{n^{2}+2 n-4}{n} & -1 & -1 & \frac{4}{n} & -1 \\
	-1 & -1 & \frac{n^{2}+2 n-4}{n} & \frac{4}{n} & -1 & -1 \\
	-1 & -1 & \frac{4}{n} & \frac{n^{2}+2 n-4}{n} & -1 & -1 \\
	\frac{4}{n} & -1 & -1 & -1 & -1 & \frac{n^{2}+2 n-4}{n}
	\end{bmatrix}
	\end{align}
	On the other hand, noting that $F^{\tau(1)\tau(2)}_{i_{1}i_{2}} = 0$, then $F^{\tau(1)\tau(2)}_{i_{1}i_{2}} - F^{\tau(1)\tau(2)}_{i_{2}i_{1}} = 0$. which implies that
	\begin{align}
	1 + \bar{x}^{\tau(1)\tau(2)}_{\tau(1)\tau(2)} - \bar{x}^{\tau(2)\tau(1)}_{\tau(1)\tau(2)} = 0, \tau(1) < \tau(2).
	\end{align}
	\begin{align}
	\bar{x}^{\tau(1)\tau(2)}_{\alpha \beta} - \bar{x}^{\tau(2)\tau(1)}_{\alpha \beta} = 0, \alpha < \beta, \alpha \beta \neq \tau(1) \tau(2),\tau(1) <  \tau(2).
	\end{align}
	In the similar manner we can rewrite it by the definition of $\bar{x}$ that
	\begin{align}
	W = {} & \begin{bmatrix}
	n & 1 & 1 & -1 & -1 & 0 \\
	1 & n & 1 & 1 & 0 & -1 \\
	1 & 1 & n & 0 & 1 & 1 \\
	-1 & 1 & 0 & n & 1 & -1 \\
	-1 & 0 & 1 & 1 & n & 1 \\
	0 & -1 & 1 & -1 & 1 & n
	\end{bmatrix} \begin{bmatrix}
	x^{12}_{12} - x^{21}_{12} & x^{12}_{13} - x^{21}_{13} & \cdots &  x^{12}_{34} - x^{21}_{34} \\
	x^{13}_{12} - x^{31}_{12} & x^{1
		3}_{13} - x^{31}_{13} & \cdots &  x^{13}_{34} - x^{31}_{43} \\ \vdots & \vdots & \cdots & \vdots \\
	x^{34}_{12} - x^{43}_{12} & x^{34}_{13} - x^{43}_{13} & \cdots &  x^{34}_{34} - x^{43}_{34}
	\end{bmatrix} \notag \\ = {} & D \tilde{\tilde{X}} =-I
	\end{align}
	where
	\begin{align}
	W =
	\begin{bmatrix}
	\bar{x}^{12}_{12} - \bar{x}^{21}_{12} & \bar{x}^{12}_{13} - \bar{x}^{21}_{13} & \cdots &  \bar{x}^{12}_{34} - \bar{x}^{21}_{34} \\
	\bar{x}^{13}_{12} - \bar{x}^{31}_{12} & \bar{x}^{1
		3}_{13} - \bar{x}^{31}_{13} & \cdots &  \bar{x}^{13}_{34} - \bar{x}^{31}_{43} \\ \vdots & \vdots & \cdots & \vdots \\
	\bar{x}^{34}_{12} - \bar{x}^{43}_{12} & \bar{x}^{34}_{13} - \bar{x}^{43}_{13} & \cdots &  \bar{x}^{34}_{34} - \bar{x}^{43}_{34}
	\end{bmatrix}
	\end{align}
	We can solve it as following
	\begin{align}
	\tilde{\tilde{X}} =	-\begin{bmatrix}
	\frac{n}{n^{2}-4} & \frac{1}{4-n^{2}} & \frac{1}{4-n^{2}} & \frac{1}{n^{2}-4} & \frac{1}{n^{2}-4} & 0 \\
	\frac{1}{4-n^{2}} & \frac{n}{n^{2}-4} & \frac{1}{4-n^{2}} & \frac{1}{4-n^{2}} & 0 & \frac{1}{n^{2}-4} \\
	\frac{1}{4-n^{2}} & \frac{1}{4-n^{2}} & \frac{n}{n^{2}-4} & 0 & \frac{1}{4-n^{2}} & \frac{1}{4-n^{2}} \\
	\frac{1}{n^{2}-4} & \frac{1}{4-n^{2}} & 0 & \frac{n}{n^{2}-4} & \frac{1}{4-n^{2}} & \frac{1}{n^{2}-4} \\
	\frac{1}{n^{2}-4} & 0 & \frac{1}{4-n^{2}} & \frac{1}{4-n^{2}} & \frac{n}{n^{2}-4} & \frac{1}{4-n^{2}} \\
	0 & \frac{1}{n^{2}-4} & \frac{1}{4-n^{2}} & \frac{1}{n^{2}-4} & \frac{1}{4-n^{2}} & \frac{n}{n^{2}-4}
	\end{bmatrix}
	\end{align}
	Finally we have
	\begin{align}
	X^{\tau(1)\tau(2)} = \begin{bmatrix}
	a & b & b & c & c & d \\
	b & a & b & b & d & c \\
	b & b & a & d & b & b \\
	c & b & d & a & b & c \\
	c & d & b & b & a & b \\
	d & c & b & c & b & a\end{bmatrix}, \tau(1) < \tau(2)
	\end{align}
	and
	\begin{align}
	X^{\tau(2)\tau(1)} = \begin{bmatrix}
	e & c & c & b & b & d \\
	c & e & c & c & d & b \\
	c & c & e & d & c & c \\
	b & c & d & e & c & b \\
	b & d & c & c & e & c \\
	d & b & c & b & c & e\end{bmatrix}, \tau(1) < \tau(2)
	\end{align}
	where
	\begin{align}
	a = -\frac{n^3+4 n^2-4}{(n-2) n (n+2) (n+4)}, b =  \frac{n+3}{(n-2)(n+2)(n+4)}.
	\end{align}
	\begin{align}
	c = \frac{-1}{(n-2)(n+2)(n+4)}, d = \frac{-2}{(n-2)n(n+4)},e = \frac{4}{(n-2) n (n+2) (n+4)}.
	\end{align}
	Now we consider the equations that $\bar{y}^{\tau(1)\tau(2)} = 0$, it is easy to check that $\bar{y}^{12} = \bar{y}^{34} = 0$, $\bar{y}^{13} = \bar{y}^{24} = 0$, $\bar{y}^{14} = \bar{y}^{23} = 0$. Then we only need to consider the following equations
	\begin{align}
	2(ny^{12} + y^{13} + y^{14}) = \frac{n+2}{(n-2)(n+4)} t^{12} - \frac{2}{(n+4)(n-2)}(t^{13} + t^{14})
	\end{align}
	\begin{align}
	2(y^{12} + n y^{13} + y^{14}) = \frac{n+2}{(n-2)(n+4)} t^{13} - \frac{2}{(n+4)(n-2)}(t^{12} + t^{14})
	\end{align}
	\begin{align}
	2(y^{12} + y^{13} + n y^{14}) = \frac{n+2}{(n-2)(n+4)} t^{14} - \frac{2}{(n+4)(n-2)}(t^{12} + t^{13})
	\end{align}
	We can solve it as following
	\begin{align}
	y^{12} = \frac{(n^2 + 3n + 6)t^{12}}{(n-2)(n-1)(n+2)(n+4)} - \frac{3n+2}{(n-2)(n-1)(n+2)(n+4)}(t^{13}+t^{14})
	\end{align}
	\begin{align}
	y^{13} = \frac{(n^2 + 3n + 6)t^{13}}{(n-2)(n-1)(n+2)(n+4)} - \frac{3n+2}{(n-2)(n-1)(n+2)(n+4)}(t^{12}+t^{14})
	\end{align}
	\begin{align}
	y^{14} = \frac{(n^2 + 3n + 6)t^{14}}{(n-2)(n-1)(n+2)(n+4)} - \frac{3n+2}{(n-2)(n-1)(n+2)(n+4)}(t^{12}+t^{13})
	\end{align}
	Finally the minimal norm tensor is
	\begin{align}
	F_{i_{1}i_{2}i_{3}i_{4}} = {} &  T_{i_{1}i_{2}i_{3}i_{4}} - \frac{2}{n(n-2)(n+4)} \sum_{\sigma \in \varphi^{\prime}_{1}(4)} T^{\sigma(3)\sigma(4)}_{i_{\sigma(1)}i_{\sigma(2)}} \delta_{i_{\sigma(3)}i_{\sigma(4)}} \notag \\ + {} & \frac{2}{n(n-2)(n+4)(n+2)}\sum_{\sigma \in \varphi^{\prime}_{1}(4)} T^{\sigma(3)\sigma(4)}_{i_{\sigma(4)}i_{\sigma(3)}} \delta_{i_{\sigma(1)}i_{\sigma(2)}} \notag \\ - {} & \frac{n^3+4n^2-4}{2n(n-2)(n+4)(n+2)}\sum_{\sigma \in \varphi^{\prime}_{1}(4)} T^{\sigma(3)\sigma(4)}_{i_{\sigma(3)}i_{\sigma(4)}}\delta_{i_{\sigma(1)}i_{\sigma(2)}} \notag \\ + {} & \frac{n+3}{(n-2)(n+4)(n+2)} \sum_{\sigma \in \varphi^{\prime}_{1}(4)}(T^{\sigma(3)\sigma(4)}_{i_{\sigma(3)}i_{\sigma(1)}}\delta_{i_{\sigma(2)}i_{\sigma(4)}} + T^{\sigma(3)\sigma(4)}_{i_{\sigma(1)}i_{\sigma(4)}}\delta_{i_{\sigma(2)}i_{\sigma(3)}}) \notag \\ - {} & \frac{1}{(n-2)(n+2)(n+4)}\sum_{\sigma \in \varphi^{\prime}_{1}(4)}(T^{\sigma(3)\sigma(4)}_{i_{\sigma(1)}i_{\sigma(3)}}\delta_{i_{\sigma(2)}i_{\sigma(4)}} + T^{\sigma(3)\sigma(4)}_{i_{\sigma(4)}i_{\sigma(1)}}\delta_{i_{\sigma(2)}i_{\sigma(3)}}) \notag \\ + {} &  \frac{n^2+3n+6}{4(n-2)(n+4)(n+2)(n-1)}\sum_{\sigma \in \varphi^{\prime}_{1}(4)}t^{\sigma(1)\sigma(2)}\delta_{i_{\sigma(1)}i_{\sigma(2)}} \delta_{i_{\sigma(3)}i_{\sigma(4)}} \notag \\ - {} &  \frac{3n+2}{2(n-2)(n+4)(n+2)(n-1)}\sum_{\sigma \in \varphi^{\prime}_{1}(4)}t^{\sigma(1)\sigma(3)}\delta_{i_{\sigma(1)}i_{\sigma(2)}} \delta_{i_{\sigma(3)}i_{\sigma(4)}} \label{eq2.45}
	\end{align}
	where $\varphi^{\prime}_{1}(4) = \lbrace \sigma \in \varphi(4) \mid \sigma(3) < \sigma(4) \rbrace$. By a direct computation we can get the minimal norm is \begin{align}
	\mid \mid F_{i_{1}i_{2}i_{3}i_{4}} \mid \mid^{2} = {} &  \mid \mid T_{i_{1}i_{2}i_{3}i_{4}} \mid \mid^{2} - \frac{2}{n(n-2)(n+4)} \sum_{\sigma \in \varphi^{\prime}_{2}(4)} T^{\sigma(3)\sigma(4)}_{i_{\sigma(1)}i_{\sigma(2)}} (T^{\sigma(1)\sigma(2)}_{i_{\sigma(1)}i_{\sigma(2)}} + T^{\sigma(1)\sigma(2)}_{i_{\sigma(2)}i_{\sigma(1)}}) \notag \\ + {} & \frac{4}{n(n-2)(n+4)(n+2)}\sum_{\sigma \in \varphi^{\prime}_{2}(4)} T^{\sigma(3)\sigma(4)}_{i_{\sigma(4)}i_{\sigma(3)}} T^{\sigma(3)\sigma(4)}_{i_{\sigma(3)}i_{\sigma(4)}} \notag \\ - {} & \frac{n^3+4n^2-4}{2n(n-2)(n+4)(n+2)}\sum_{\sigma \in \varphi^{\prime}_{1}(4)} T^{\sigma(3)\sigma(4)}_{i_{\sigma(3)}i_{\sigma(4)}}T^{\sigma(3)\sigma(4)}_{i_{\sigma(3)}i_{\sigma(4)}} \notag \\ + {} & \frac{2(n+3)}{(n-2)(n+4)(n+2)} (\sum_{\sigma \in \varphi^{\prime}_{3}(4)}T^{\sigma(3)\sigma(4)}_{i_{\sigma(3)}i_{\sigma(1)}}T^{\sigma(1)\sigma(3)}_{i_{\sigma(1)}i_{\sigma(3)}} + \sum_{\sigma \in \varphi^{\prime}_{4}(4)} T^{\sigma(3)\sigma(4)}_{i_{\sigma(1)}i_{\sigma(4)}} T^{\sigma(1)\sigma(4)}_{i_{\sigma(1)}i_{\sigma(4)}}) \notag \\ - {} & \frac{2}{(n-2)(n+2)(n+4)}(\sum_{\sigma \in \varphi^{\prime}_{3}(4)}T^{\sigma(3)\sigma(4)}_{i_{\sigma(1)}i_{\sigma(3)}}T^{\sigma(1)\sigma(3)}_{i_{\sigma(1)}i_{\sigma(3)}} + \sum_{\sigma \in \varphi^{\prime}_{4}(4)} T^{\sigma(3)\sigma(4)}_{i_{\sigma(4)}i_{\sigma(1)}}T^{\sigma(1)\sigma(4)}_{i_{\sigma(1)}i_{\sigma(4)}}) \notag \\ + {} &  \frac{n^2+3n+6}{4(n-2)(n+4)(n+2)(n-1)}\sum_{\sigma \in \varphi^{\prime}_{1}(4)}(t^{\sigma(1)\sigma(2)})^{2} \notag \\ - {} &  \frac{3n+2}{2(n-2)(n+4)(n+2)(n-1)}\sum_{\sigma \in \varphi^{\prime}_{1}(4)}t^{\sigma(1)\sigma(3)}t^{\sigma(1)\sigma(2)}
	\end{align}
	where $\varphi^{\prime}_{2}(4) = \lbrace \sigma \in \varphi(4) \mid \sigma(3) < \sigma(4), \sigma(1) < \sigma(2) \rbrace$, $\varphi^{\prime}_{3}(4) = \lbrace \sigma \in \varphi(4) \mid \sigma(3) < \sigma(4), \sigma(1) < \sigma(3) \rbrace$ and $\varphi^{\prime}_{4}(4) = \lbrace \sigma \in \varphi(4) \mid \sigma(3) < \sigma(4), \sigma(1) < \sigma(4) \rbrace$. \\
	2) If $\lbrace T^{\alpha \beta},(T^{\alpha \beta})^{T}, I \rbrace$ (where $(1 \le \alpha < \beta) \le 4$) are linearly dependent, then the effect variables are less than before since we can use the relation of linear dependence, we can prove that the minimal norm tensor solved by the similar disscussion (in $\lbrace T^{\alpha \beta},(T^{\alpha \beta})^{T}, I \rbrace$ are linearly independent) can be written by equation \eqref{eq2.45} if we use the relation of linear dependence again in equation \eqref{eq2.45}.  This proves the theorem.
\end{proof}
\begin{proof}[Proof of Corollary \ref{cor1.5}]
	Letting $T_{i_{1}i_{2}i_{3}i_{4}} = T_{ijkl} = R_{ijkl}$, by a direct computation
	\begin{align}
	T^{(1,2)}_{ij} = T^{(1,3)}_{ij} = 0, T^{(1,2)(3,4)} = 0, T^{(1,3)(2,4)} = -T^{(1,4)(2,3)} = - R.
	\end{align}
	\begin{align}
	T^{(1,3)}_{ij} = T^{(2,4)}_{ij} = R_{ij}, T^{(1,4)}_{ij} = T^{(2,3)}_{ij} = -R_{ij}.
	\end{align}
	Then we have
	\begin{align}
	F_{ijkl} = {} & R_{ijkl}  - \frac{1}{n-2}[R_{jl}\delta_{ik} - R_{jk}\delta_{il} - R_{il}\delta_{jk} + R_{ik}\delta_{jl}] \notag \\ & + \frac{R}{(n-1)(n-2)}(\delta_{jl}\delta_{ik} - \delta_{jk}\delta_{il})
	\end{align}
	This is Weyl conformal curvature tensor. This proves the corollary.
\end{proof}
	
	\nocite{*}
	\bibliographystyle{IEEEtran}
	\bibliography{minimal}

\end{document}